\documentclass[11pt,a4paper,reqno]{amsart}
\usepackage[margin=2cm]{geometry}
\usepackage{amssymb}
\usepackage{amsmath,amsthm}
\usepackage{amsfonts}
\usepackage{mathrsfs}
\usepackage{color}
\usepackage{cite}
\usepackage{hyperref}
\hypersetup{
    colorlinks = true,%
    citecolor = [rgb]{0.0,0.0,0.9},
    filecolor=black,%
    linkcolor = [rgb]{0.65,0.0,0.0},%
    anchorcolor = red,
    pagecolor = red,
    urlcolor= [rgb]{0.65,0.0,0.0}
}
	\title[]{On a critical Kirchhoff-type problem}\newcommand{\N}{{\mathbb N}}
	\newcommand{\R}{{\mathbb R}}
    \def\d{{\,{\rm d}}x}
	
	\newcommand{\ds}{\displaystyle}
	\newtheorem{theor}{Theorem}[section]

	\newtheorem{rem}{Remark}[section]
\numberwithin{equation}{section}
\newtheorem{innercustomthm}{\rm \textbf{Theorem}}
\newenvironment{customthm}[1]
{\renewcommand\theinnercustomthm{#1}\innercustomthm}
{\endinnercustomthm}
\author{Francesca Faraci }

\email{ffaraci@dmi.unict.it}

\address{Department of Mathematics and Computer Science, University of Catania,
	Catania, Italy}

\author{Csaba Farkas}

\email{farkas.csaba2008@gmail.com \& farkascs@ms.sapientia.ro}

\address{Department of Mathematics and Computer Science, Sapientia Hungarian University of Transylvania, Tg. Mure\c s, Romania}
\subjclass[2010]{35J20, 35J60}
\keywords{Kirchhoff type problem, critical nonlinearity, sequentially weakly lower semicontinuity, Palais Smale condition, exterior domains.}
\begin{document}
	\maketitle

	\begin{abstract} In the present paper we study  a Kirchhoff type  problem involving the critical Sobolev exponent. We give sufficient conditions for the sequentially weakly lower semicontinuity and the Palais Smale property of the energy functional associated to the problem.
	\end{abstract}
	
	
	
	\section{Introduction}
In the present paper we deal with the following Kirchhoff type problem involving a critical term	
$$\left\{
	\begin{array}{ll}
	- M\left(\displaystyle\int_{\Omega}|\nabla u|^p \d  \right)\Delta_p u=
	|u|^{p^*-2}u, & \hbox{ in } \Omega \\ \\
	u=0, & \hbox{on } \partial \Omega
	\end{array}
	\right. \eqno{(\mathcal{P})}
	$$
	where $\Omega$ is an open connected set of $\R^N$ with smooth boundary, $1<p<N$, $p^*=\frac{pN}{N-p}$ is the critical Sobolev exponent, $M:[0,+\infty[\to [0,+\infty[$ is a continuous function.

	The nonlocal operator  $\ds M\left(\int_{\Omega}|\nabla u|^p \d  \right)\Delta_p u$  generalizes the term $\ds \left(a+b\int_{\Omega}|\nabla u|^2\right)\Delta u$ of the Kirchhoff equation,
 proposed in 1883 as a model for describing the transversal oscillation of a strectched strings (\cite{Kirchhoff}).

Nonlocal problems received wide attention   after the pioneering work of  Lions (\cite{JLL}), where a functional analysis approach was implemented to study problems arising in the theory of evolutionary boundary value problems of mathematical physics.
After the work of Alves, Correa and Figueiredo (\cite{ACF}),
the existence and multiplicity of solutions of  Kirchhoff type problems  with critical nonlinearities in  bounded or unbounded domains (or even in the whole space) have been studied by a number of authors by employing different techniques as  variational methods, genus theory,  the Nehari manifold, the Ljusternik--Schnirelmann category theory (see for instance \cite{AFP, CF,Fan,F, FP} and the references therein). We mention also the recent works \cite{H2, N0}, where an application of the Lions' Concentration Compactness principle (\cite{L}) ensures the Palais Smale condition of the energy functional, a key property  for the application of the classical Mountain Pass Theorem.
	
	In the present paper, generalizing the recent work  \cite{FF}, we claim to show that the interaction between the Kirchhoff  and the critical term leads to some  variational properties of the energy functional as the sequentially weakly lower semicontinuity and the Palais Smale condition.
	
	 An application to a Kirchhoff type problem on exterior domains is given. More precisely, we combine our results with a recent minimax theory by Ricceri (\cite{R}) to prove the existence of two solutions for a suitable perturbation of $(\mathcal{P})$.

	\bigskip

	 Before stating our results, let us introduce some notations. Let $\widehat M:[0,+\infty[\to[0,+\infty[$ be the primitive of the function $M$, defined by
	\[\widehat M(t)=\int_0^t M(s)\, \mathrm{d}s.\]
	We endow the Sobolev and the Lebesgue spaces $W^{1,p}_0(\Omega)$ and $L^q(\Omega)\ (1\leq q\leq p^*)$ with the
	classical norms $$\|u\|=\left(\displaystyle \int_{\Omega }|\nabla u|^p \d \right)^{\frac{1}{p}}\ \ \mbox{ and }\ \ \|u\|_q=\left(\int_\Omega |u|^q \d\right)^{\frac{1}{q}}$$ respectively, and denote by $S_N$ the embedding constant of $W^{1,p}_0(\Omega)\hookrightarrow L^{p^*}(\Omega)$, i.e.
	\[\|u\|^p_{p^*}\leq S_N^{-1} \|u\|^p, \quad \mbox{for every } \ u\in W^{1,p}_0(\Omega). \]

	Let $\mathcal E:W^{1,p}_0(\Omega)\to\R$ be the energy functional associated to  the problem $(\mathcal{P})$,  defined by
	\[\mathcal E(u)=\frac{1}{p}\widehat M(\|u\|^p)-\frac{1}{p^*}\|u\|^{p^*}_{p^*}, \quad \mbox{for every }\  u \in W^{1,p}_0(\Omega),\] whose derivative at $u \in W^{1,p}_0(\Omega)$ is given by
	\[\mathcal E'(u)(v)=M(\|u\|^p)\int_\Omega |\nabla u|^{p-2}\nabla v\nabla u \d- \int_\Omega|u|^{p^*-2}uv \d, \quad \mbox{for every }\  v \in W^{1,p}_0(\Omega).\]
Define the constant 	\begin{equation}\label{theconstantcp}
c_p=\left\{
\begin{array}{ll}
(2^{p-1}-1)^\frac{p^*}{p}\frac{p}{p^*}{S}_N^{-\frac{p^*}{p}}, & \hbox{if }  \ p\geq 2 \\ \\
2^{2p^*-1-\frac{p^*}{p}}\frac{p}{p^*}{S}_N^{-\frac{p^*}{p}}, & \hbox{if } \ 1<p<2
\end{array}
\right.
\end{equation}

To prove  the lower semicontinuity property we require  the following   assumptions on $\widehat M$:
\begin{itemize}
	\item[$i)$] \label{subadditivity}$\widehat M(t+s)\geq \widehat M(t)+\widehat M(s)$ for every $t,s\in [0,+\infty[$;
	
	\item[$ii)$] \label{growthofM}$\ds \inf_{t> 0} \ \displaystyle\frac{\widehat M(t)}{t^{p^*\over p}}\geq c_p. $
\end{itemize}

 Our first result allows $\Omega$ to be an unbounded subset of $\R^N$, due to a general inequality valid for $p\geq 2$.
	
	\begin{theor}\label{semicontinuity1}
	Let $\Omega$ be an open connected set with smooth boundary, $p\geq 2$.
	
	If $i)$ and $ii)$ hold, then $\mathcal E$ is sequentially weakly lower semicontinuous on $W^{1,p}_0(\Omega)$.
	\end{theor}

The same conclusion holds for $1<p<2$, but in our proof we need the boundedness of $\Omega$. It remains an open question if the property still holds for a general domain $\Omega$.
		\begin{theor}\label{semicontinuity2}
			
				Let $\Omega$ be a bounded open connected set with smooth boundary, $1<p<2$.
				
				If $i)$ and $ii)$ hold, then $\mathcal E$ is sequentially weakly lower semicontinuous on $W^{1,p}_0(\Omega)$.
		\end{theor}

	In order to ensure  the Palais--Smale property we need the following condition on $M$: \begin{itemize}
		\item[$iii)$] \label{conditionforPS} $\ds \inf_{t> 0} \ \displaystyle \frac{ M(t)}{t^{\frac{p^*}{p}-1}}> {S}_N^{-\frac{p^*}{p}}. $
	\end{itemize}
		\begin{theor}\label{PS}
		Let $\Omega$ be an open connected set with smooth boundary, $p>1$.
		
		If $iii)$ holds, then
			 $\mathcal E$ satisfies the Palais--Smale property in $W^{1,p}_0(\Omega)$.
		\end{theor}

		\section{Proof of Theorems \ref{semicontinuity1} and \ref{semicontinuity2}}

			Fix $u \in W^{1,p}_0(\Omega)$ and let $\{u_n\} \subset W^{1,p}_0(\Omega)$ such that $u_n\rightharpoonup u$ in $W^{1,p}_0(\Omega)$.
			Assume by contradiction that \begin{equation}\label{contr}l\equiv \liminf_{n\to\infty}\mathcal E(u_n)<\mathcal E(u).\end{equation} Then, there exists a subsequence $\{u_{n_j}\}$ such that
			$$\lim_{j\to\infty} \mathcal E(u_{n_j})=l.$$
						
			Since  $\mathcal E$ is strongly continuous, $\{u_{n_j}\}$ can not be strongly convergent to $u$ in $W^{1,p}_0(\Omega)$, that is $\ds \limsup_{j\to\infty}\|u_{n_j}-u\|=L>0$. Thus, there exists a a subsequence (still denoted by $\{u_{n_j}\}$), such that
		\[	\lim_{j\to\infty}\|u_{n_j}-u\|=L>0.\]
			The above limit implies also that
			\begin{equation}\label{limit1}
			\lim_{j\to\infty} \|u_{n_j}\|>\|u\|.
			\end{equation}
			We have,
			\begin{align*}
			\mathcal{E}(u_{n_j})-\mathcal{E}(u) =&\frac{1}{p}\left(\widehat M(\|u_{n_j}\|^p)-\widehat M(\|u\|^p)\right)-\frac{1}{p^*}\left(\|u_{n_j}\|_{p^*}^{p^*}-\|u\|_{p^*}^{p^*}\right).
			\end{align*}
			Let us recall that the Br\'ezis-Lieb lemma  implies
			
			\begin{equation}\label{BL}
			\|u_{n_j}\|_{p^*}^{p^*}-\|u\|_{p^*}^{p^*}=\|u_{n_j}-u\|_{p^*}^{p^*}+o(1).
			\end{equation}
			We distinguish now the two cases $p\geq 2$ and $1<p<2$.
			
			\medskip
			
			\noindent {\bf End of proof of Theorem \ref{semicontinuity1}: $p\geq 2$}.
			From \cite[Lemma 2]{Li}, the following inequality holds true:
			\begin{equation}\label{RN}
			|x|^p-|y|^p\geq \frac{1}{2^{p-1}-1}|x-y|^p +p\langle x-y, |y|^{p-2 }y\rangle  \qquad \mbox{for every} \ x,y\in\R^N.
			\end{equation}
			Thus,
			 \[\|u_{n_j}\|^p\geq \|u\|^p+ \frac{1}{2^{p-1}-1}\|u_{n_j}-u\|^p +p \int_{\Omega}|\nabla u|^{p-2}\nabla u\nabla (u_{n_j}-u)\d.\]
 Since $\{u_{n_j}\}$ weakly converges to  $u$,
		
	 \[\frac{1}{2^{p-1}-1}\|u_{n_j}-u\|^p+p\int_{\Omega}|\nabla u|^{p-2}\nabla u\nabla (u_{n_j}-u)\d\to \frac{1}{2^{p-1}-1}L^p>0,\]  and for $j$ large enough, one has $\ds \frac{1}{2^{p-1}-1}\|u_{n_j}-u\|^p+p\displaystyle\int_{\Omega}|\nabla u|^{p-2}\nabla u\nabla (u_{n_j}-u)\d > 0$.
Employing the monotonicity of  $\widehat M$ ensured by  $i)$, and \eqref{BL}, together with $i)$ and $ii)$ we obtain
			\begin{align*}
			\mathcal{E}(u_{n_j})-\mathcal{E}(u)=&\frac{1}{p}
			\widehat M\left(\|u\|^p+\frac{1}{2^{p-1}-1}\|u_{n_j}-u\|^p+p\displaystyle\int_{\Omega}|\nabla u|^{p-2}\nabla u\nabla (u_{n_j}-u)\d\right)-\frac{1}{p}\widehat M(\|u\|^p)\\&-\frac{1}{p^*}\|u_{n_j}-u\|_{p^*}^{p^*}+o(1) \\ \overset{i)}{\geq}&
			\frac{1}{p}\widehat M\left(\frac{1}{2^{p-1}-1}\|u_{n_j}-u\|^p+p\displaystyle\int_{\Omega}|\nabla u|^{p-2}\nabla u\nabla (u_{n_j}-u)\d\right) \\&-\frac{{S}_N^{-\frac{p^*}{p}}}{p^*}\|u_{n_j}-u\|^{p^*}+o(1)\\ \overset{ii)}{\geq} & \frac{c_p}{p} \left(\frac{1}{2^{p-1}-1}\|u_{n_j}-u\|^p+p\displaystyle\int_{\Omega}|\nabla u|^{p-2}\nabla u\nabla (u_{n_j}-u)\d\right)^{\frac{p^*}{p}}\\&-\frac{{S}_N^{-\frac{p^*}{p}}}{p^*}
			\|u_{n_j}-u\|^{p^*}+o(1).
			\end{align*}
		Passing to the limit in the above estimate, one has that
		$$l-\mathcal E(u)\geq \frac{{S}_N^{-\frac{p^*}{p}}}{p^*} L^{p^*}-\frac{{S}_N^{-\frac{p^*}{p}}}{p^*} L^{p^*}=0,$$ which contradicts \eqref{contr}.	
		
		\medskip
		
		\noindent {\bf End of proof of Theorem \ref{semicontinuity2}: $1<p<2$}. For $k \geq 1$, we consider the following two auxiliary functions $T_k,R_k:\mathbb{R}\to \mathbb{R}$ given by $$T_k(s)=\left\{
		\begin{array}{ll}
		-k,& \hbox{ if } s< -k \\
		s, & \hbox{ if } -k\leq s \leq k\\
		k, & \hbox{ if } s> k,
		\end{array}
		\right. $$
		and $R_k=\mathrm{Id}_{\mathbb{R}}-T_k$, i.e.,
		$$R_k(s)=\left\{
		\begin{array}{ll}
		s+k,& \hbox{ if } s> -k \\
		0, & \hbox{ if } -k\leq s \leq k\\
		s-k, & \hbox{ if } s> k,
		\end{array}
		\right. $$	
so, for every $v \in W_0^{1,p}(\Omega)$
	\begin{equation}
	\label{dec}
	\| v\|^p=\| T_k(v)\|^p+\|R_k(v)\|^p,
	\end{equation}
	and
	\[\lim_{k \to \infty}\|T_k(v)\|^p=\|v\|^p \qquad \mbox{and} \qquad \lim_{k \to \infty}\|R_k(v)\|=0. \] Also, for every $k\in\N$ (see \cite{MM})
	\begin{equation}\label{liminf}
	\liminf_{n \to \infty}\|T_k(u_n)\|^p\geq \|T_k(u)\|^p.
	\end{equation}
	From \eqref{dec} and the elementary inequality
	\[	\|R_k(u_n)\|^p\geq \frac{1}{2^{p-1}}\|R_k(u_n)- R_k(u)\|^p-\| R_k(u)\|^p,\]	 we obtain that
	\begin{align*}
	\|u_{n_j}\|^p-\|u\|^p=&\|T_k(u_{n_j})\|^p  + \|R_k(u_{n_j})\|^p -\|T_k(u)\|^p-\|R_k(u)\|^p\\
	\geq & \|T_k(u_{n_j})\|^p-\|T_k(u)\|^p +\frac{1}{2^{p-1}}\|R_k(u_n)- R_k(u)\|^p-2\| R_k(u)\|^p.
	\end{align*}


Moreover we point out that
\begin{align*}
\|u_{n_j}-u\|_{p^*}^{p^{*}}&=\int_{\Omega}|T_k(u_{n_j})-T_k(u)+R_k(u_{n_j})-R_k(u)|^{p^*}\d\nonumber \\ &\leq 2^{p^*-1}\int_{\Omega}|T_k(u_{n_j})-T_k(u)|^{p^{\ast}}\d+2^{p^*-1}\int_{\Omega}|R_k(u_{n_j})-R_k(u)|^{p^{*}}\d,
\end{align*}
and, since $\Omega$ is bounded, we can apply  Lebesgue dominated convergence  theorem to get, for every $k\in\N$,
 $$\lim_{j \to \infty} \int_{\Omega}|T_k(u_{n_j})-T_k(u)|^{p^{*}}\d= 0$$ and
\[\|u_{n_j}-u\|_{p^*}^{p^{*}} \leq 2^{p^*-1} {S}_N^{-\frac{p^*}{p}}\|R_k(u_{n_j})-R_k(u)\|^{p^*}+o_j(1),\]
where $o_j(1)\to 0$ as $j\to \infty$.

\noindent We consider two subcases:

a) $\ds \lim_{k\to \infty}\lim_{j\to \infty}\|R_k(u_{n_j})-R_k(u)\|=0$.
In such case, we have that
\begin{align*}
\lim_j\|u_{n_j}-u\|_{p^*}^{p^*}=& \lim_k\lim_j\|u_{n_j}-u\|_{p^*}^{p^*}\\
\leq & 2^{p^*-1} {S}_N^{-\frac{p^*}{p}}\lim_k\lim_j\|R_k(u_{n_j})-R_k(u)\|^{p^*}=0.
\end{align*}
Thus, together with \eqref{BL} and assumption $i)$ we get
\begin{align*}
	\mathcal{E}(u_{n_j})-\mathcal{E}(u)=&\frac{1}{p}
	\widehat M\left(\|u_{n_j}\|^p-\|u\|^p+\|u\|^p\right)-\frac{1}{p}\widehat M(\|u\|^p)-\frac{1}{p^*}\|u_{n_j}-u\|_{p^*}^{p^*}+o_j(1) \\ \geq	
	&\frac{1}{p}
	\widehat M\left(\|u_{n_j}\|^p-\|u\|^p\right)-\frac{1}{p^*}\|u_{n_j}-u\|_{p^*}^{p^*}+o_j(1),
\end{align*}
and passing to the limit,
\[l-\mathcal{E}(u)\geq 0,\] which contradicts \eqref{contr}.

b) $\ds \lim_{k\to \infty}\lim_{j\to \infty}\|R_k(u_{n_j})-R_k(u)\|>0$. Thus, there exist  $\alpha>0$ and $\bar k\in\N$ with the property: for every $k>\bar k$ there exists $j_k\in\N$ such that for $j>j_k$ one has
\[\|R_k(u_{n_j})-R_k(u)\|>\alpha.\]
Choosing eventually  bigger $\bar k$ and for $k>\bar k$, bigger $j_k$ one can assume that
\[\|T_k(u_{n_j})\|^p-\|T_k(u)\|^p +  \frac{1}{2^{p-1}}\|R_k(u_{n_j})- R_k(u)\|^p-2\| R_k(u)\|^p>0,\] so that
putting things together, and using
	\eqref{BL}, assumptions $i)$ and $ii)$,  for $k>\bar k$ and $j>j_k$,
	\begin{align}\label{esti}
	\mathcal{E}(u_{n_j})-\mathcal{E}(u)=&\frac{1}{p}
	\widehat M\left(\|u_{n_j}\|^p-\|u\|^p+\|u\|^p\right)-\frac{1}{p}\widehat M(\|u\|^p)-\frac{1}{p^*}\|u_{n_j}-u\|_{p^*}^{p^*}+o_j(1)\nonumber \\ \overset{i)}{\geq}	
	&\frac{1}{p}
	\widehat M\left(\|u_{n_j}\|^p-\|u\|^p\right)-\frac{1}{p^*}\|u_{n_j}-u\|_{p^*}^{p^*}+o_j(1) \nonumber \\ \geq&
	  \frac{1}{p} \widehat M\left(\|T_k(u_{n_j})\|^p-\|T_k(u)\|^p +  \frac{1}{2^{p-1}}\|R_k(u_{n_j})- R_k(u)\|^p-2\| R_k(u)\|^p \right)
	\nonumber \\& -\frac{2^{p^*-1}}{p^*}\|R_k(u_{n_j})-R_k(u)\|_{p^*}^{p^*}+o_j(1)
\nonumber 	\\
	\overset{ii)}{\geq}& \frac{c_p}{p}\left(\|T_k(u_{n_j})\|^p-\|T_k(u)\|^p +  \frac{1}{2^{p-1}}\|R_k(u_{n_j})- R_k(u)\|^p-2\| R_k(u)\|^p \right)^{\frac{p^*}{p}}
\nonumber 	\\&-{S}_N^{-\frac{p^*}{p}}\frac{2^{p^*-1}}{p^*}\|R_k(u_{n_j})-R_k(u)\|^{p^*}+o_j(1)\nonumber \\ \geq
		 & \frac{c_p \|R_k(u_{n_j})-R_k(u)\|^{p^*}}{p}	\left[\left(\frac{\|T_k(u_{n_j})\|^p-\|T_k(u)\|^p}{\|R_k(u_{n_j})-R_k(u)\|^p}+\frac{1}{2^{p-1}}-\frac{2\| R_k(u)\|^p}{\alpha^p} \right)^{\frac{p^*}{p}}\right]\nonumber \\&-{S}_N^{-\frac{p^*}{p}}\frac{2^{p^*-1}}{p^*}\|R_k(u_{n_j})-R_k(u)\|^{p^*} +o_j(1).
		\end{align}	
		Since for  fixed $k>\bar k$, by \eqref{liminf},
$$\liminf_{j\to\infty}	\frac{\|T_k(u_{n_j})\|^p-\|T_k(u)\|^p}{\|R_k(u_{n_j})-R_k(u)\|^p}\geq 0 \qquad \mbox{and} \qquad \lim_{j\to\infty}\frac{o_j(1)}{\|R_k(u_{n_j})-R_k(u)\|^{p^*}}=0,$$
we have
\begin{align*}
\liminf_{k\to\infty}\liminf_{j\to\infty}&\left[\left(\frac{\|T_k(u_{n_j})\|^p-\|T_k(u)\|^p}{\|R_k(u_{n_j})-R_k(u)\|^p}+\frac{1}{2^{p-1}}-\frac{2\| R_k(u)\|^p}{\alpha^p} \right)^{\frac{p^*}{p}}-{S}_N^{-\frac{p^*}{p}}\frac{2^{p^*-1}p}{p^*c_p}\right] \\ & \geq\liminf_{k\to\infty}\left(\frac{1}{2^{p-1}}-\frac{2\| R_k(u)\|^p}{\alpha^p}\right)^{\frac{p^*}{p}}-{S}_N^{-\frac{p^*}{p}}\frac{2^{p^*-1}p}{p^*c_p}\\&\overset{\eqref{theconstantcp}}{\geq} \left(\frac{1}{2^{p^*-\frac{p^*}{p}}}-{S}_N^{-\frac{p^*}{p}}\frac{2^{p^*-1}p}{p^*c_p}\right)=0
\end{align*}
Thus, for every $\varepsilon>0$ there exist   $\tilde k\in\N$  such that  for every $k>\tilde k$ there exists $j_k\in\N$ such that for $j>j_k$ one has
$$\left[\left(\frac{\|T_k(u_{n_j})\|^p-\|T_k(u)\|^p}{\|R_k(u_{n_j})-R_k(u)\|^p}+\frac{1}{2^{p-1}}-\frac{2\| R_k(u)\|^p}{\alpha^p} \right)^{\frac{p^*}{p}}-{S}_N^{-\frac{p^*}{p}}\frac{2^{p^*-1}p}{p^*c_p}\right]>-\varepsilon.$$
Moreover, by the boundedness of $\{u_{n_j}\}$ it follows that there exists a constant $M>0$ such that  $\|R_k(u_{n_j})-R_k(u)\|^{p^*}\leq M$, so that we obtain

$$\frac{c_p \|R_k(u_{n_j})-R_k(u)\|^{p^*}}{p}\left[\left(\frac{\|T_k(u_{n_j})\|^p-\|T_k(u)\|^p}{\|R_k(u_{n_j})-R_k(u)\|^p}+\frac{1}{2^{p-1}}-\frac{2\| R_k(u)\|^p}{\alpha^p} \right)^{\frac{p^*}{p}}-{S}_N^{-\frac{p^*}{p}}\frac{2^{p^*-1}p}{p^*c_p}\right]>-\varepsilon M\frac{c_p}{p}.$$
This means that passing to the $\liminf$ in the right hand side of \eqref{esti},
we obtain
$$
l-\mathcal E(u)=\liminf_{k\to\infty}\lim_{j\to\infty} (\mathcal{E}(u_{n_j})-\mathcal{E}(u))\geq0,
$$	which contradicts \eqref{contr}.

The proof is complete.
\qed

	\begin{rem}\label{with constant}
	{\rm	From the proof of the above theorems, it immediately follows that   for $\mu\geq 1$, the functional $\mathcal E_\mu:W^{1,p}_0(\Omega)\to \R$ defined by
		\[\mathcal E_\mu(u)=\frac{{\mu}}{p}\widehat M(\|u\|^p)-\frac{1}{p^*}\|u\|^{p^*}_{p^*}  \quad \mbox{for every }\  u \in W^{1,p}_0(\Omega),\] is sequentially weakly lower semicontinuous in  $W^{1,p}_0(\Omega)$.}
		\end{rem}
			\begin{rem}
				{\rm A comparison with \cite[Lemma 2.1]{FF} is in order. \\					
				The above result extends to more general Kirchhoff operator the result of \cite[Lemma 2.1]{FF} where the simple case $M(t)= a+bt $ when $p=2$ is considered. The sequential weak lower semicontinuity of $\mathcal E$ in such case is proved when (see also the proof of \cite[Lemma 2.1]{FF})
 \begin{itemize}
 \item $N=4$, $a\geq 0, b>0$ with $b\geq C_1(N)$,
 \item $N>4$, $a> 0, b>0$ with $ a^{\frac{N-4}{2}} b\geq C_1(N)$, 
 \end{itemize}
 where \[
					C_1(N)=
					\begin{cases}\ds
					\frac{4(N-4)^{\frac{N-4}{2}}}{N^{\frac{N-2}{2}}S_{N}^{\frac{N}{2}}} & N>4\\
					\ds S_{4}^{-2} & N=4,
					\end{cases}\]
Notice that  when $p=2$, assumption $ii)$ in Theorem \ref{semicontinuity1} reads as
					\[\inf_{t> 0} \ \displaystyle\frac{\widehat M(t)}{t^{2^*\over 2}}\geq \frac{2}{2^*}{S}_N^{-\frac{2^*}{2}}, \]
					and it is fulfilled by the particular case investigated in \cite{FF}. 
					\\
					 Notice also that, with respect to \cite{FF}, our result holds also  for $N=3$.}
			\end{rem}
		\begin{rem} {\rm When $p\geq 2$ we have a "better" estimate of the constant $c_p$. This reason is due to the fact that inequality \eqref{RN} is no longer valid when $1<p<2$ and we employ more rough estimates which carry along some extra constants. }
	\end{rem}
		\begin{rem}
			{\rm When $p=2,$  assumption $ii)$ is equivalent to the following sign property for $\mathcal E$:
			\[\mathcal E(u)\geq 0 \quad \mbox{for every } \ u\in W^{1,2}_0(\Omega).\]
			If $ii)$ does not hold, there exists  $\bar t>0$ such that
$\ds \frac{1}{2}\widehat M(\bar t\,)<\frac{{S}_N^{-\frac{2^*}{2}}}{2^*}{\bar t^{2^*\over 2}}$.
Let $\{u_n\}$ be a minimizing sequence for $S_N$. Then, it is weakly convergent and it has a subsequence $\{u_{n_k}\}$ strongly converging. Let $c>0$ such that $\ds c\lim_{k\to \infty}\|u_{n_k}\|=\sqrt{\bar t}$. Then,
\begin{eqnarray*}
\lim_{k\to\infty}\mathcal E(cu_{n_k})&=&\lim_{k\to\infty}\left[\frac{1}{2}\widehat M(\|c u_{n_k}\|^2)-\frac{1}{2^*}\|c u_{n_k}\|_{2^*}^{2^*}\right]\\&=&
\lim_{k\to\infty}\left[\frac{1}{2}\widehat M(\|c u_{n_k}\|^2)-\frac{{S}_N^{-\frac{2^*}{2}}}{2^*}\|c u_{n_k}\|^{2^*}\right]\\
&=&\frac{1}{2}\widehat M(\bar t)-\frac{{S}_N^{-\frac{2^*}{2}}}{2^*}{\bar t^{2^*\over 2}}<0.
\end{eqnarray*}
The reverse implication follows from the Sobolev embedding. It is clear that we can not expect such equivalence  for $p>2$ since inequality \eqref{RN} is far from being optimal.
}
			
			\end{rem}
			
			\section{Proof of Theorem \ref{PS}}

	 		Let $\{u_{n}\}$ be a  Palais Smale sequence for $\mathcal E$, that is
	 		\[
	 		\begin{cases}
	 		\mathcal{E}(u_{n})\to c\\
	 		\mathcal{E}'(u_{n})\to0
	 		\end{cases}\mbox{as }n\to\infty.
	 		\] We claim that $\{u_{n}\}$ admits a  strongly convergent subsequence in $W^{1,p}_0(\Omega)$.

	 		Let us first notice that $\mathcal E$ is coercive. Indeed, let $k$ be a positive constant such that  $k>{S}_N^{-\frac{p^*}{p}}$ and $M(t)\geq k t^{\frac{p^*}{p}-1}$ for every $t\geq 0$. Then,
	 		$\widehat M(t)\geq \frac{p}{p^*} k t^{\frac{p^*}{p}}$ for every $t\geq 0$ and  $$\mathcal E(u)\geq \frac{1}{p^*}\left(k-{S}_N^{-\frac{p^\star}{p}}\right)\|u\|^{p^*}, \  \mbox{for every}  \ u\in W^{1,p}_0(\Omega),$$ and coercivity of $\mathcal E$ follows at once.
	 		
	 		Then, the (PS) sequence   $\{u_{n}\}$ is bounded and there exists $u\in W^{1,p}_0(\Omega)$ such that (up to a subsequence)
	 		
	 		\begin{align*}
	 		u_{n} & \rightharpoonup u\mbox{ in }W^{1,p}_0(\Omega),\\
	 		u_{n} & \to u\mbox{ in }L^{q}_{\rm loc}(\Omega),\ q\in[1,p^{*}),\\
	 		u_{n} & \to u\mbox{ a.e. in }\Omega.
	 		\end{align*}
	 		Using the second Concentration Compactness lemma of Lions \cite{L}, there exist an at most countable index set $J$,
	 		a set of points $\{x_{j}\}_{j\in J}\subset\overline\Omega$ and two families of positive
	 		numbers $\{\eta_{j}\}_{j\in J}$, $\{\nu_{j}\}_{j\in J}$ such that
	 		\begin{align*}
	 		|\nabla u_{n}|^{p} & \rightharpoonup d\eta\geq|\nabla u|^{p}+\sum_{j\in J}\eta_{j}\delta_{x_{j}},\\
	 		|u_{n}|^{p^*} & \rightharpoonup d\nu=|u|^{p^*}+\sum_{j\in J}\nu_{j}\delta_{x_{j}},
	 		\end{align*}
	 		(weak star convergence in the sense of measures), where $\delta_{x_{j}}$ is the Dirac mass concentrated at
	 		$x_{j}$ and such that
	 		$$		S_{N}  \nu_{j}^{\frac{p}{p^*}}\leq\eta_{j} \qquad \mbox{for every $j\in J$}.$$
	 		Next, we will prove that the index set $J$ is empty. Arguing
	 		by contradiction, we may assume that there exists a $j_{0}$ such
	 		that $\nu_{j_{0}}\neq0$. Consider now, for $\varepsilon>0$ a non negative  cut-off function $\phi_\varepsilon$ such that
	 		\begin{align*}
	 		&\phi_{\varepsilon}  =1\mbox{ on }B(x_{0},\varepsilon),\\
	 		&\phi_{\varepsilon}  =0\mbox{ on } \Omega\setminus B(x_{0},2\varepsilon),\\
	 		&|\nabla\phi_{\varepsilon}|  \leq\frac{2}{\varepsilon}.
	 		\end{align*}
	 		It is clear that the sequence $\{u_{n}\phi_{\varepsilon}\}_{n}$ is
	 		bounded in $W^{1,p}_0(\Omega)$,  so that
	 		\[
	 		\lim_{n\to\infty}\mathcal{E}'(u_{n})(u_{n}\phi_{\varepsilon})=0.
	 		\]
	 		Thus
	 		\begin{align}\label{calc 1}
	 		o(1) & =M(\|u_n\|^p)\int_{\Omega}|\nabla u_{n}|^{p-2}\nabla u_n\nabla(u_{n}\phi_{\varepsilon})\d-\int_{\Omega}|u_{n}|^{p^*}
	 		\phi_{\varepsilon}\d\nonumber \\
	 		& =M(\|u_n\|^p)\left(\int_{\Omega}|\nabla u_{n}|^{p}\phi_{\varepsilon}\d+\int_{\Omega}u_{n}|\nabla u_n|^{p-2}\nabla u_{n}\nabla\phi_{\varepsilon}\d\right)-\int_{\Omega}|u_{n}|^{p^*}\phi_{\varepsilon}\d.
	 		\end{align}
	 		Using  H\"{o}lder inequality one has
	 		
	 		\begin{eqnarray*}
	 			\left|\int_\Omega u_{n}|\nabla u_n|^{p-2}\nabla u_{n}\nabla\phi_{\varepsilon}\d\right|&=&\left|\int_{B(x_{0},2\varepsilon)}u_{n}|\nabla u_n|^{p-2}\nabla u_{n}\nabla\phi_{\varepsilon}\right|\d\\ &\leq& \left(\int_{B(x_{0},2\varepsilon)}|\nabla u_n|^p\d\right)^\frac{1}{p}
	 			\left(\int_{B(x_{0},2\varepsilon)}|u_n\nabla \phi_\varepsilon|^{p'}\d\right)^\frac{1}{p'}\\
	 			&\leq&  C \left(\int_{B(x_{0},2\varepsilon)}|u_n\nabla \phi_\varepsilon|^{p'}\d\right)^\frac{1}{p'},
	 		\end{eqnarray*}
	 		where $p'$ is the conjugate of $p$.
	 		Since $$\lim_{n\to\infty}\int_{B(x_{0},2\varepsilon)}|u_n\nabla \phi_\varepsilon|^{p'}\d=\int_{B(x_{0},2\varepsilon)}|u\nabla \phi_\varepsilon|^{p'}\d,$$ and
	 		\begin{eqnarray*}
	 			\left(\int_{B(x_{0},2\varepsilon)}|u\nabla \phi_\varepsilon|^{p'}\d\right)^\frac{1}{p'}&\leq &
	 			\left(\int_{B(x_{0},2\varepsilon)} |u|^{p^*}\d\right)^\frac{1}{p^*}
	 			\left(\int_{B(x_{0},2\varepsilon)}|\nabla \phi_\varepsilon|^{\frac{pN}{p(N+1)-2N}}\d \right)^\frac{p(N+1)-2N}{pN}\\
	 			&\leq& C \left(\int_{B(x_{0},2\varepsilon)} |u|^{p^*}\d\right)^\frac{1}{p^*} \longrightarrow 0,\  \hbox{as} \ \varepsilon\to 0
	 		\end{eqnarray*}
	 	by the boundedness of the sequence $\{M(\|u_n\|^p)\}_n$	we get
	 		
	 		\[
	 		\lim_{\varepsilon\to0}\lim_{n\to\infty}M(\|u_{n}\|^{p})\left|\int_\Omega u_{n}|\nabla u_n|^{p-2}\nabla u_{n}\nabla\phi_{\varepsilon}\d\right|=0.
	 		\]	
	 		
	 		Moreover, as $0\leq \phi_\varepsilon\leq 1$,
	 		\begin{eqnarray*}
	 			\lim_{n\to\infty}M(\|u_{n}\|^{p})\int_{\Omega}|\nabla u_{n}|^{p}\phi_{\varepsilon}\d&\geq&
	 		k\lim_{n\to\infty}\left(\int_{B(x_{0},2\varepsilon)}|\nabla u_{n}|^{p}\phi_{\varepsilon}\d\right)^{\frac{p^*}{p}}\\&\geq&
	 			k\left(\int_{B(x_{0},2\varepsilon)}|\nabla u|^{p}\phi_{\varepsilon}\d+ \eta_{j_0}\right)^{\frac{p^*}{p}}.
	 		\end{eqnarray*}
	 		Also, $\ds \int_{B(x_{0},2\varepsilon)}|\nabla u|^{p}\phi_{\varepsilon}\d\to 0$ as $\varepsilon\to 0$, so
	 		\[\lim_{\varepsilon\to 0}\lim_{n\to\infty}M(\|u_{n}\|^{p})
	 		\int_{\Omega}|\nabla u_{n}|^{p}\phi_{\varepsilon}\d \geq k\eta_{j_0}^{\frac{p^*}{p}}.\]
	 		
	 		Finally,
	 		\begin{align*}
	 		\lim_{\varepsilon\to0}\lim_{n\to\infty}\int_\Omega|u_{n}|^{p^*}\phi_{\varepsilon}\d & =\lim_{\varepsilon\to0}\int_{B(x_{0},2\varepsilon)} |u|^{p^*}\phi_{\varepsilon}\d+\nu_{j_{0}}=\nu_{j_{0}}.
	 		\end{align*}
	 		Summing up the above outcomes, from \eqref{calc 1}
	 		one obtains
	 		\[
	 		0 \geq k\eta_{j_0}^{\frac{p^*}{p}}-\nu_{j_0} =\left(k-S_{N}^{-\frac{p^*}{p}}\right)\eta_{j_{0}}^{\frac{p^*}{p}}\geq 0.
	 		\]
	 	
	 			 		Therefore $\eta_{j_{0}}=0,$ which is a contradiction. Such conclusion  implies
	 		that $J$ is empty, that is
	 		\[\lim_{n\to\infty}\int_{\Omega}|u_n|^{p^*}\d= \int_{\Omega}|u|^{p^*}\d\]
	 		and the uniform convexity of $L^{p^*}(\Omega)$ implies that 		\[
	 		u_{n}\to u\mbox{ in }L^{p^*}(\Omega).
	 		\]
	 		Since $\{u_n-u\}$ is bounded in $W^{1,p}_{0}(\Omega)$,
	 		\[
	 		\lim_{n\to\infty}\mathcal{E}'(u_{n})(u_{n}-u)=\lim_{n\to\infty}\left[M(\|u_n\|^p)\int_{\Omega}|\nabla u_n|^{p-2}\nabla u_n \nabla (u_n-u)\d-\int_\Omega |u_n|^{p^*-2}u_n(u_n-u)\d\right]=0.
	 		\]
	 	From H\"{o}lder inequality, 	
	 		\[\left|\int_\Omega |u_n|^{p^*-2}u_n(u_n-u)\d\right|\leq
	 		\left(\int_\Omega|u_n|^{p^*}\d\right)^{\frac{pN-N+p}{pN}}\left(\int_\Omega|u_n-u|^{p^*}\d\right)^\frac{1}{p^*},
	 		\]
	 		so we deduce that
	 		 \[
	 		\lim_{n\to\infty}M(\|u_n\|^p)\left|\int_{\Omega}|\nabla u_n|^{p-2}\nabla u_n \nabla (u_n-u)\d\right|=0.
	 		\]
	 		We claim that
	 	\begin{equation}\label{strong}	\lim_{n\to\infty}\int_{\Omega}|\nabla u_n|^{p-2}\nabla u_n \nabla (u_n-u)\d=0.
	 	\end{equation}
	 		If $\ds \limsup_{n\to\infty}M(\|u_n\|^p)>0$, then, \eqref{strong} follows at once. If $\ds \lim_{n\to\infty}M(\|u_n\|^p)=0$, then, by $iii)$, we obtain that $u_n\to 0$ strongly in $W^{1,p}_{0}(\Omega)$ and \eqref{strong} holds true also in this case.
	 		
	 		Putting together  \eqref{strong} with the limit \[\lim_{n\to\infty}\int_{\Omega}|\nabla u|^{p-2}\nabla u \nabla (u_n-u)\d=0,\] we obtain
	 		\[\lim_{n\to\infty}\int_{\Omega}\left(|\nabla u_n|^{p-2}\nabla u_n-|\nabla u|^{p-2}\nabla u\right) \nabla (u_n-u)\d=0,\]
	 		which implies at once  that $u_n\to u$ strongly in $W^{1,p}_{0}(\Omega)$.
     \qed
	 		
	 			\begin{rem}\rm
	 				The above result extends the result of \cite[Lemma 2.2]{FF} where the simple case $M(t)= a+bt $ when $p=2$ is considered. The Palais Smale property for $\mathcal E$ in such case is proved when (see also the proof of \cite[Lemma 2.2]{FF})
 \begin{itemize}
 \item $N=4$, $a\geq 0, b>0$ with $b> C_2(N)$,
 \item $N>4$, $a> 0, b>0$ with $ a^{\frac{N-4}{2}} b>C_2(N)$,
 \end{itemize}
 where \[C_2(N)=\begin{cases}
	 				\ds\frac{2(N-4)^{\frac{N-4}{2}}}{(N-2)^{\frac{N-2}{2}}S_{N}^{\frac{N}{2}}} & N>4\\
	 				\ds {S_{4}^{-2}} & N=4.
	 				\end{cases}.
	 				\]
Notice that  when $p=2$, assumption $iii)$ reads as
	 				\[\inf_{t> 0} \ \displaystyle \frac{ M(t)}{t^{\frac{2^*}{2}-1}}> {S}_N^{-\frac{2^*}{2}}. \]
	 				and it is fulfilled by the particular case investigated in \cite{FF}.  It is worth mentioning  that Palais--Smale property for $\mathcal E$ when $M(t)=a+bt$ and $p=2$, was first proved by Hebey on compact Riemannian manifolds (\cite{H2}).	 		\end{rem}

\section{An application}

In this section, we provide an application of Theorem \ref{semicontinuity1} to the following Kirchhoff problem on an exterior domain:
$$\left\{
\begin{array}{ll}
- M\left(\displaystyle\int_{\Omega}|\nabla u|^p \d \right)\Delta_p u=\lambda(
u^{p^*-1}+u^{q-1}), & \hbox{ in } \Omega \\
u\geq 0 & \hbox{in } \Omega \\
u=0, & \hbox{on } \partial \Omega
\end{array}
\right.\eqno{(\mathcal{P}_\lambda)}
$$
where $\Omega=\R^N\setminus B(0,R)$ for some positive $R$, $2\leq p<q<p^*$, and $\lambda>0$.

As far a we know there are no contributions  about subcritical and critical Kirchhoff equations on exterior
unbounded domains beside \cite{FM} where an existence result is proved via a careful analysis of Palais Smale sequences and an application of the mountain pass theorem.

We will prove a multiplicity result for problem $(\mathcal{P}_\lambda)$ by employing an abstract   well--posedness result for a class of constrained minimization problem which is derived by a general minimax theory of Ricceri (\cite{R}). Let us first recall the following definition:
if $X$ is a topological space, $C\subset X$, $f:X\to\R$, we say that the problem of minimizing $f$ over $C$ is well-posed if the following two conditions hold:
\begin{itemize}
	\item[-] the restriction of $f$ to $C$ has a unique global minimum, say $x_0$;
	\item[-] every sequence $\{x_n\}$ in $C$ such that $\ds \lim_{n\to \infty} f (x_n) = \min_C f$, converges to $x_0$.
\end{itemize}

\begin{customthm}{\textbf{A} (\cite[Theorem 1]{R})}\label{RicceriJogo}
	Let $X$ be a Hausdorff topological space,  $a>0$, $\Phi,\Psi:X\to\R$, two functions such that, for each $\mu> a$, $\mu\Psi+\Phi$ has sequentially compact sub-level sets and admits a unique global minimum in $X$. Denote by $\mathscr{M}_a$ the set of global minima of $a \Psi+\Phi$ and assume that
 $$\inf_{X}\Psi <\inf_{\mathscr{M}_a} \Psi$$ (if $\ds \mathscr{M}_a=\emptyset$, put $\ds \inf_{\mathscr{M}_a} \Psi=+\infty$).
	Then, for each $\ds r\in ]\inf_{X}\Psi,\inf_{\mathscr{M}_a} \Psi[$ the problem of minimizing  $\Phi$ on $\Psi^{-1}(r)$ is well posed.
\end{customthm}

Our multiplicity result reads as follows:
\begin{theor}\label{application}
	Let $\Omega=\R^N\setminus B(0,R)$ for some $R>0$,  $2\leq p<q<p^*$,  $M:[0,+\infty[\to [0,+\infty[$ a continuous function such that if $\widehat M$ denotes its primitive, the following conditions hold:
	\begin{itemize}
			\item[$i)$] $\ds\widehat  M(t+s)\geq \widehat  M(t)+\widehat  M(s)$, for every $t,s\in [0,+\infty[$;
			
			\item[$ii)$] $\ds\inf_{t> 0} \ \displaystyle\frac{\widehat  M(t)}{t^{p^*\over p}}> c_p$;
		\item[$iv)$] $\ds\lim_{t\to 0} \ \displaystyle\frac{\widehat  M(t)}{t^{q\over p}}=0$,
		\end{itemize}
		where $c_p$ is from \eqref{theconstantcp}.
		
		Then, there  exists   $\lambda^*\in]0,1[$ such that the problem $(\mathcal P_{\lambda^*})$ has two nontrivial solutions.
	\end{theor}
	\begin{proof} Denote by $X=W^{1,p}_{0, {\rm rad}}(\Omega)$, the subspace of $W^{1,p}_{0}(\Omega)$ consisting of radial functions. It is well known that such space is  embedded into $L^{r}(\Omega)$ continuously for $r\in [p,p^*]$, and compactly  for $r\in ]p,p^*[$ (\cite{L1}).
		We apply Theorem \hyperref[RicceriJogo]{\textbf{A}}, to the space $X$ endowed with the weak topology, choosing $a=1$ and  functions \[\Psi(u)=\frac{1}{p}\widehat M(\|u\|^p)  \quad \hbox{and} \quad \Phi(u)=-\frac{1}{p^*}\|u^+\|_{p^*}^{p^*}-\frac{1}{q}\|u^+\|_{q}^{q}.\]
		For $\mu\geq 1$, define in $X$ the functional
		\[\mathcal F_\mu(u)=\mu \Psi(u)+\Phi(u)=\frac{\mu}{p}\widehat  M(\|u\|^p)-\frac{1}{p^*}\|u^+\|^{p^*}_{p^*}-
			\frac{1}{q}\|u^+\|^{q}_{q}, \quad \hbox{for every} \ u\in X.\] Thus,  if $\mathcal E_\mu(u)=\frac{{\mu}}{p}\widehat M(\|u\|^p)-\frac{1}{p^*}\|u\|^{p^*}_{p^*}$, then
			\begin{equation}
			\label{FandE}
			\mathcal{F}_\mu(u)=\mathcal{E}_\mu(u)-\frac{1}{p^*}\|u^-\|_{p^*}^{p^*}-\frac{1}{q}\|u^+\|_q^q.
			\end{equation}
			
			 From  $ii)$,  $\mathcal F_\mu$ is coercive.
			The sequential weak lower semicontinuity of  $\mathcal F_\mu$  follows from Theorem \ref{semicontinuity1}.
				Indeed, if $\{u_n\}\subset X$ is  weakly convergent  to $u\in X$, 			
			then $u_n^-\rightharpoonup	 u^-$ in $X$, and in particular in $L^{p^*}(\Omega)$, and  $u_n^+\to u^+$ (strongly)  in $L^{q}(\Omega)$, therefore from Remark \ref{with constant} and from the lower semicontinuity of the norm we have that
		\begin{eqnarray}\label{3lsc}
			&&\mathcal E_\mu(u)\leq \liminf_n \mathcal E_\mu(u_n),\nonumber\\
			&&\|u^-\|^{p^*}_{p^*}\leq  \liminf_n \|u_n^-\|^{p^*}_{p^*},\\
			&&\|u^+\|^{q}_{q}= \lim_n\|u_n^+\|^{q}_{q}.\nonumber
		\end{eqnarray}
	Therefore, on account of \eqref{FandE} and \eqref{3lsc} we have that  $$\liminf_{n\to \infty}(\mathcal F_{\mu}(u_n)- \mathcal F_{\mu}(u))\geq 0,$$ hence, $\mu \Psi+{\Phi}$ has sequentially weakly compact sub-level sets.
			
				Fix now a positive function $\bar u\in X$ and $\ds \varepsilon\in \left]0,\frac{p\|\bar u\|_q^q}{q\|\bar u\|^q}\right[ $. From assumption $iv)$ it follows that   there exists  $\delta>0$ such that $\widehat M(t)<\varepsilon t^{\frac{q}{p}}$ for $0<t\leq \delta$. If  $\ds\rho\in \left]0, \frac{\delta}{\|\bar u\|}\right[$ then,
				\begin{eqnarray*}
					\mathcal F_{1}(\rho\bar u)<\rho^q\left(\frac{1}{p}\varepsilon\| \bar u\|^q-
					\frac{1}{q}\|\bar u\|^{q}_{q}\right)-\frac{1}{p^*}\rho^{p^*}\|\bar u\|^{p^*}_{p^*}<0.
			\end{eqnarray*}

		Thus, if we denote by $\mathscr{M}_1$ the set of global minima of $\mathcal F_1=\Phi+\Psi$ (which is non empty), $0\notin \mathscr{M}_1$.
		
		We claim that $\ds \inf_{\mathscr{M}_1}\Psi>0$. Indeed, if $\ds \inf_{\mathscr{M}_1}\Psi=0$, there would exist a sequence $\{u_n\}$ in  $\mathscr{M}_1$ such that $\Psi(u_n)\to 0$, that is $u_n\to 0$, but then, by continuity of $\Psi+\Phi$,
		\[(\Psi+\Phi)(0)=\min_X (\Psi+\Phi),\] which is in contradiction with $0\notin \mathscr{M}_1$. Thus,
		\[\inf_X \Psi=0<\inf_{\mathscr{M}_1}\Psi.\] Let $\ds r\in]0,\inf_{\mathscr{M}_1}\Psi[$. Assume by contradiction that $\Phi$ has a global  minimum  $u_0$ on $$\Psi^{-1}(r)=\{u\in X: \|u\|^p=(\widehat M)^{-1}(pr)\}.$$ Thus, by the Lagrange multiplier rule, there exists $\sigma_0\leq 0$ such that $$\Phi'(u_0)=\sigma_0\Psi'(u_0),$$
		or
		\[\sigma_0\int_{\Omega}|\nabla u_0|^{p-2}\nabla u_0\nabla v\d =-
		\int_{\Omega}({u_0}_+^{p^*-1}+{u_0}_+^{q-1})v\d, \quad \mbox{for every} \ v\in X
		\]
		Thus,  $\sigma_0<0$, $u_0\geq0$ and plugging $v=u_0$ in the above equality we deduce that
		\[\int_{\Omega}|\nabla u_0|^{p}\d =-\frac{1}{\sigma_0}
		\int_{\Omega}({u_0}^{p^*}+{u_0}^{q})\d.
		\]
		On the other hand, by the Poho\v{z}aev equality,  one has
		\begin{eqnarray*}
			\frac{p-1}{p}\int_{\partial\Omega} |\nabla u_0|^p \sigma\cdot \nu \mathrm{d}\sigma&=&-\frac{N}{\sigma_0}\int_{\Omega }\left(\frac{u_0^{p^*}}{p^*}+\frac{u_0^{q}}{q}\right)\d-\frac{N-p}{p}\int_{\Omega}|\nabla u_0|^p\d=\\
			&=&
			\frac{N}{\sigma_0}\left(-\frac{1}{q}+\frac{1}{p^*}\right)\int_{\Omega}u_0^q \d>0,
		\end{eqnarray*}
		against the fact that  the left hand side is non positive (due to the shape of the domain $\Omega$). Thus, the problem of minimizing $\Phi$ on $\Psi^{-1}(r)$ is not well posed and by Theorem \hyperref[RicceriJogo]{\textbf{A}}, we conclude that there exists $\mu^*>1$ such that $\mu^*\Psi+\Phi$ has  two distinct global minima in $X$. By the Palais principle of symmetric criticality \cite{Pa} these minima are critical points of $\mathcal F_{\mu^*}$, i.e. solutions of $(\mathcal P_{\lambda^*})$ with $\lambda^*=\frac{1}{\mu^*}$.
		\end{proof}
\begin{rem}
{\rm In order to prove  that  the problem of minimizing $\Phi$ on $\Psi^{-1}(r)$ is not well posed, one could  prove either that  $\Phi$ has  two distinct global minima on $\Psi^{-1}(r)$ or that $\Phi$ has no global minima on the level set. This is the first application of Theorem \hyperref[RicceriJogo]{\textbf{A}} where such conclusion is obtained by showing that through the Pohozaev inequality  $\Phi$ has no global minima on $\Psi^{-1}(r)$. For contributions where the claim is achieved via the existence of two global minima for the constrained problem we mention for instance \cite{R2,R3}}.\end{rem}
\begin{rem}
{\rm Theorem \ref{application} applies for instance to the simple case $N=4$, $\Omega =\R^4\setminus B(0,R)$,  $M(t)=bt$ where $b>2c_2$ and $p=2<q<4$.}
\end{rem}
	

\section{Concluding remarks}
In the present paper we presented some energy properties of the energy functional associated to a critical Kirchhoff problem with an application  to a nonlocal problem on an exterior domain. We believe that such properties can be used in different settings to establish, by the means of critical point theory, existence and multiplicity results for perturbations of $(\mathcal P)$. We formulate some open problems which could be object of   forthcoming investigations.
\begin{enumerate}
	\item Theorem \ref{semicontinuity1} allows us to consider critical Kirchhoff equation on unbounded domain while   the boundedness of $\Omega$ is needed in the proof of Theorem \ref{semicontinuity2}. We conjecture that the sequential weak lower semicontinuity property holds true for general domains even when $1<p<2$.
	\item In Theorem \ref{application} we proved the existence of  $\mu^*>1$, such that the functional $\mathcal F_{\mu^*}$ (see \eqref{FandE}) has two distinct global minima in $X=W^{1,p}_{0, {\rm rad}}(\Omega)$. Can we say that  such points are global minima of $\mathcal F_{\mu^*}$ in $W^{1,p}_{0}(\Omega)$?
	\item In Theorem \ref{application}, under the additional  assumption $iii)$ which ensures the  Palais--Smale condition for $\mathcal F_{\mu^*}$, the energy functional admits a third critical point in $X$ as it follows by \cite{PucciSerrin}. Can we exclude, under additional  assumptions on $M$, that such third solution for $(\mathcal P_{\lambda^*})$ is trivial?
\end{enumerate}
\section*{Acknowledgment}
\noindent The authors would like to thank Prof. Ricceri for illuminating discussion. \\
The work of F. Faraci has been supported by the Universit\`{a} degli Studi di Catania, "Piano della Ricerca 2016/2018 Linea di intervento 2". F. Faraci is member of the Gruppo Nazionale per l'Analisi Matematica, la Probabilit\`{a}
 e le loro Applicazioni (GNAMPA) of the Istituto Nazionale di Alta Matematica (INdAM). C. Farkas has been supported by the National Research, Development and Innovation Fund of Hungary, financed under the K\_18 funding scheme, Project No. 127926 and the Sapientia Foundation – Institute for Scientific Research, Project No. 17/11.06.2019.

\end{document}